\documentclass[a4paper,11pt]{article}
\usepackage{graphicx}
\usepackage[T1]{fontenc}
\usepackage[utf8]{inputenc}
\usepackage[english]{babel}
\usepackage{amsmath,amsthm,amssymb,amsfonts}
\usepackage{enumitem}
\usepackage{lmodern}
\usepackage{color}
\usepackage{cite}
\usepackage{algorithm}
\usepackage{algorithmic}
\usepackage{subcaption}
\usepackage{tikz}
\usepackage{pgfplots}
\usepackage{mathtools}

\usepackage{geometry}
\definecolor{darkgreen}{rgb}{0.1,0.5,0.00}

\usepackage{hyperref}
\hypersetup{
	colorlinks=true,        
	linkcolor=blue,         
	citecolor=darkgreen,    
	urlcolor=blue           
}

\newtheorem{theorem}{Theorem}[section]
\newtheorem{lemma}[theorem]{Lemma}
\newtheorem{proposition}[theorem]{Proposition}
\newtheorem{corollary}[theorem]{Corollary}
\theoremstyle{remark}
\newtheorem{remark}[theorem]{Remark}

\newcommand{\n}[1]{\left\|#1 \right\|}
\renewcommand{\a}{\alpha}
\renewcommand{\b}{\beta}
\newcommand{\la}{\lambda}
\newcommand{\e}{\varepsilon}

\newcommand{\R}{\mathbb R}
\newcommand{\N}{\mathbb N}
\newcommand{\lr}[1]{\left\langle #1\right\rangle}

\DeclareMathOperator{\dom}{dom}
\DeclareMathOperator{\gra}{gra}
\newcommand{\Hilbert}{\mathcal{H}}
\newcommand{\setto}{\rightrightarrows}

\title{A Forward-Backward Splitting Method for Monotone Inclusions Without Cocoercivity}
\author{Yura Malitsky\thanks{Institute for Numerical and Applied Mathematics,
                             University of G\"ottingen,
                             37083 G\"ottingen, Germany.
                             Email:~\href{mailto:y.malitsky@gmail.com}{y.malitsky@gmail.com}}
       \and
       Matthew K. Tam\footnotemark[1]\hspace{1ex}${}^,$\hspace{0.1ex}\thanks{School of Mathematics \& Statistics,
                             The University of Melbourne,
                             Parkville VIC 3010, Australia.
                             Email:~\href{href:matthew.tam@unimelb.edu.au}{matthew.tam@unimelb.edu.au}}
}

\begin{document}
\maketitle

\begin{abstract}
    In this work, we propose a simple modification of the
    forward-backward splitting method for finding a zero in the sum of
    two monotone operators. Our method converges under the same
    assumptions as Tseng's forward-backward-forward method, namely, it
    does not require cocoercivity of the single-valued
    operator. Moreover, each iteration only uses one forward
    evaluation rather than two as is the case for Tseng's
    method. Variants of the method incorporating a linesearch,
    relaxation and inertia, or a structured three operator inclusion are
    also discussed.
\end{abstract}


\medskip


\section{Introduction}\label{s:intro}
In this work, we propose an algorithm for finding a zero in the sum of two monotone operators in a (real) Hilbert space $\Hilbert$. Specifically, we consider the monotone inclusion problem
  \begin{equation}\label{eq:zero 2sum}
   \text{find}~x\in\Hilbert\text{ such that }0\in (A+B)(x),
  \end{equation}
where $A\colon \Hilbert\setto\Hilbert$ and $B\colon \Hilbert\to\Hilbert$ are (maximally) monotone operators with $B$ (locally) Lipschitz continuous such that $(A+B)^{-1}(0)\neq\varnothing$. Inclusions of the form specified by \eqref{eq:zero 2sum} arise in numerous problems of fundamental importance in mathematical optimization, either directly or through an appropriate reformulation. In what follows, we provide some motivating examples.

\textbf{Convex minimization:} Consider the minimization problem
	$$ \min_{x\in\Hilbert}f(x)+g(x),$$
	where $f\colon\Hilbert\to(-\infty,+\infty]$ is proper, lower semicontinuous (lsc), convex and $g\colon\Hilbert\to\mathbb{R}$ is convex with (locally) Lipschitz continuous gradient denoted $\nabla g$. The solutions to this minimization problem are precisely the points $x\in\Hilbert$ which satisfy the \emph{first order optimality condition:}
	  \begin{equation}
	  \label{eq:constrained min}
	  0\in \left(\partial f+\nabla g\right)(x),
	  \end{equation}
	where $\partial f$ denotes the \emph{subdifferential} of $f$. Clearly \eqref{eq:constrained min} is of the form specified by \eqref{eq:zero 2sum}.

\textbf{General monotone inclusions:} Consider the inclusion problem

\begin{equation}\label{eq:zero 2sum_lin}
	  \text{find}~x\in\Hilbert_1\text{ such that }0\in (A+K^*BK)(x),
	  \end{equation}
	where $A\colon\Hilbert_1\setto\Hilbert_1$ and $B\colon\Hilbert_2\setto\Hilbert_2$ are maximally monotone operators, and $K\colon\Hilbert_1\to\Hilbert_2$ is a linear, bounded operator with adjoint $K^*$.
	As was observed in \cite{briceno2011monotone+,bot2013primal}, solving \eqref{eq:zero 2sum_lin} can be equivalently cast as the following monotone inclusion posed in the product space:
	\begin{equation}
	\label{eq:zero 2sum2}
	\text{find}~\binom{x}{y}\in\Hilbert_1\times\Hilbert_2 \text{ such that
	}\binom{0}{0}\in
	\left(
	\begin{bmatrix}
	A & 0\\
	0 & B^{-1}
	\end{bmatrix}
	+    \begin{bmatrix}
	0 & K^*\\
	-K & 0
	\end{bmatrix}\right)
	\binom{x}{y}.
	\end{equation}
	Notice that the first operator in \eqref{eq:zero 2sum2} is maximally monotone whereas the second is bounded and linear (in particular, it is Lipschitz continuous with full domain). Consequently, \eqref{eq:zero 2sum2} is also of the form specified by \eqref{eq:zero 2sum}.

    Another variant of \eqref{eq:zero 2sum} is the three operator inclusion
	\begin{equation}
	\label{eq:zero 3sum}
	\text{find}~x\in\Hilbert\text{ such that }0\in (A+B+C)(x),
	\end{equation}
	where the operators $A$ and $B$ are as before and $C\colon \Hilbert\to \Hilbert $ is $\beta$-cocoercive. Problems with this structure have been studied in \cite{davis2017three,briceno:davis}.

\textbf{Saddle point problems and variational inequalities:} Many convex optimization problems can be formulated as the \emph{saddle point problem}
	\begin{equation}
	\label{eq:saddle}
	\min_{x\in \Hilbert}\max_{y\in \Hilbert} g(x) + \Phi(x,y) - f(y),
	\end{equation}
	where $f, g\colon \Hilbert \to (-\infty, +\infty]$ are proper,
        lsc, convex functions and $\Phi\colon \Hilbert\times \Hilbert
        \to \R$ is a smooth convex-concave function.
	Problems of this form naturally arise in machine learning,
        statistics, etc., where the dual (maximization) problem comes
        from either dualizing the constraints in the primal problem or
        from using the Fenchel--Legendre transform to leverage a
        nonsmooth composite part.  Through its first-order optimality
        condition, the saddle point problem \eqref{eq:saddle} can
        expressed as the monotone inclusion
	\begin{equation}
	\label{eq:saddle-incl}
	\text{find}~\binom{x}{y}\in\Hilbert\times \Hilbert \text{ such that
	}\binom{0}{0}\in
	\binom{\partial g(x)}{\partial f(y)}
	+ \binom{\phantom{-}\nabla_x \Phi(x,y)}{-\nabla_y \Phi(x,y)},
	\end{equation}
	which is of the form specified by \eqref{eq:zero 2sum}.
	By using the definitions of the respective subdifferentials, \eqref{eq:saddle-incl} can also be expressed in terms of the \emph{variational inequality (VI)}: find $z^* = (x^*, y^*)^\top \in \Hilbert\times
	\Hilbert$ such that
	\begin{equation}
	\label{eq:vip}
	\lr{B(z^*),z-z^*} + g(x) - g(x^*) - f(y) + f(y^*) \geq 0 \quad	\forall z=\binom{x}{y} \in \Hilbert\times\Hilbert,
	\end{equation}
	 where $B(x,y) := \left( \nabla_x \Phi(x,y),-\nabla_y
             \Phi(x,y) \right)^\top$.

\bigskip

\emph{Splitting algorithms} are a class of methods which can be used
to solve \eqref{eq:zero 2sum} by only invoking each operator
individually rather than their sum directly. The individual steps
within each iteration of these methods can be divided into two
categories: \emph{forward evaluations} in which the value of a
single-valued operator is computed, and \emph{backward evaluations} in
which the \emph{resolvent} of an operator computed. Recall that the
resolvent of an operator $A$ is given by $J_{A}:=(I+A)^{-1}$ where
$I\colon \Hilbert\to\Hilbert$ denotes the identity operator.

When the resolvents of both of the involved operators can
be easily computed, there are various algorithms in the literature which are
suitable for solving \eqref{eq:zero 2sum} with $B$ not necessarily
single-valued. The best known example of such an algorithm is the
\emph{Douglas--Rachford method} \cite{lions-mercier,svaiter2011weak}. In practice,
however, it is usually not the case that both resolvents can be readily
computed and thus in order to efficiently deal with realistic
problems, it is often necessary to impose further structure on the
operators in \eqref{eq:zero 2sum}. Splitting methods which do not require computation of two
 resolvents are therefore of practical interest.

 The best-known splitting method for solving the inclusion
 \eqref{eq:zero 2sum} when $B$ is single-valued is the
 \emph{forward-backward method}, called so because each iteration
 combines one forward evaluation of $B$ with one backward evaluation
 of $A$. More precisely, the method generates a sequence according to
  \begin{equation}
  \label{eq:forward-backward}
    x_{k+1} = J_{\lambda A}(x_k-\lambda B(x_k)) \quad \forall k\in\mathbb{N},
  \end{equation}
  and converges weakly to a solution provided the operator
  $B\colon \Hilbert\to\Hilbert$ is \emph{$1/L$-cocoercive} and
  $\lambda\in (0,2/L)$.  Recall that $B\colon \Hilbert\to\Hilbert$ is
  \emph{$\beta$-cocoercive} if
 $$ \lr{x-y,B(x)-B(y)} \geq \beta \n{B(x)-B(y)}^2\quad\forall x,y\in\Hilbert. $$
 Cocoercivity of an operator is a stronger property than Lipschitz
 continuity and hence can be difficult to satisfy for general monotone
 inclusions. For instance, apart from the trivial case when $K=0$, the
 skew-symmetric operator in \eqref{eq:zero 2sum2} is never
 cocoercive. Furthermore, without cocoercivity, convergence of
 \eqref{eq:forward-backward} can only be guaranteed in the presence of
 similarly strong assumptions such as strong monotonicity of $A+B$
 \cite{chen1997convergence}, or at the cost of incorporating a
 backtracking strategy \cite{bello2015variant} (even when the
 Lipschitz constant is known).

 In order to relax the cocoercivity assumption, Tseng \cite{T2000}
 proposed a modification of the forward-backward algorithm, known as
 the \emph{Tseng's method} or the \emph{forward-backward-forward
 method}, which only requires Lipschitzness of $B$ at the expense of
 an additional forward evaluation. Applied to \eqref{eq:zero 2sum},
 Tseng's method generates sequences according to
\begin{equation}\label{eq:tseng}
\left\{\begin{aligned}
y_k     &= J_{\lambda A}(x_k-\lambda B(x_k)) \\
x_{k+1} &= y_k-\la B(y_k)+\la B(x_k) \\
\end{aligned}\right.\quad \forall k\in\mathbb{N},
\end{equation}
and converges weakly provided $B$ is $L$-Lipschitz and $\la\in(0,1/L)$.

In this work, we introduce and analyze a new method for solving
\eqref{eq:zero 2sum} which converges under the same assumptions as
Tseng's method, but whose implementation requires only one forward
evaluation per iteration instead of two. For a fixed stepsize
$\lambda>0$, the proposed scheme can be simply described as
   \begin{equation}\label{eq:the iteration intro}
   x_{k+1} = J_{\la A}\bigl(x_k - 2 \la B(x_k) + \la B(x_{k-1})\bigr)\quad\forall k\in\mathbb{N},
   \end{equation}
   and converges weakly if $B$ is $L$-Lipschitz and the stepsize is
   chosen to satisfy $\la<\frac{1}{2L}$.  We refer to this scheme as
   the \emph{forward-reflected-backward method}. It is worth noting
   that the analysis of our method is entirely different than existing
   schemes, and hence is of interest in its own right. In particular,
   the sequence generated by the method is not Fej{\'e}r monotone, although
   it does satisfy a quasi-Fej{\'e}r property \cite{combettes2001quasi}.
   Moreover, there are
   relatively few fundamentally different alternatives to Tseng's
   forward-backward-forward algorithm for solving inclusions in the
   form of \eqref{eq:zero 2sum} without cocoercivity
   \cite{johnstone2018projective,csetnek2019shadow,rieger2020backward}.

We also remark that our method is of particular interest in the
setting of the saddle point problem \eqref{eq:saddle-incl}. Indeed,
one of the first splitting techniques for solving \eqref{eq:saddle} is
the famous \emph{Arrow--Hurwicz algorithm} \cite{arrow1958studies}
which suffers from the shortcoming of requiring strict assumptions to
ensure convergence. This was remedied in late 70's when various
modification of the algorithm were proposed
\cite{korpel:76,antipin1978method,popov:80} which turned out to be
applicable not only to saddle point problems, but also to more general
variational inequalities. Note also that the simplest case of
\eqref{eq:saddle} occurs when $\Phi$ is a bilinear form and gives rise
to the popular \emph{primal-dual algorithm}, first analyzed by
Chambolle~\&~Pock~\cite{chambolle2011first}. In a recent preprint~\cite{hamedani2018primal}, a variant of this
algorithm, which can be applied when $\Phi$ is not necessarily
bilinear, was considered. Such an extension is a significant
improvement as it provides an approach to the saddle point problem
that is different from variational inequality methods.  An interesting common feature of the methods in
\cite{chambolle2011first,hamedani2018primal,malitsky2018golden,csetnek2019shadow} as
well as the one presented here is that their respective iterations
include a ``reflection term'' in which the value of an operator at the
previous point is subtracted from twice its value at the current
point.

In addition to general interest in monotone inclusions from
optimization community described above, a new surge has appeared in
machine learning research,
see~\cite{daskalakis2018training,ryu2019ode,mertikopoulos2018optimistic,gidel2018variational,mishchenko2019revisited}
and the references therein. In these works, the authors design
algorithms for training \emph{generative adversarial networks (GANs)}~\cite{goodfellow2014generative}. Although this takes the form of
a nonconvex-nonconcave min-max problem,
the main workhorses are based on classical algorithms for solving
monotone variational inequalities. Thus,
we believe that new algorithmic ideas, even for the monotone case, may have
some impact in this field as well.

\bigskip

The remainder of this paper is organized as follows. In
Section~\ref{s:FoRB splitting}, we introduce our method and prove its
convergence (Theorem~\ref{th:main}). In Section~\ref{s:linear
    convergence}, this result is refined to show that convergence is
linear whenever one of the operators is strongly monotone. In
Section~\ref{s:linesearch}, we incorporate a linesearch procedure into
the method (Theorem~\ref{th:linesearch}). In Section~\ref{s:relaxed
    inertial}, we consider a relaxed inertial version
(Theorem~\ref{th:relaxed inertial}) and, in
Section~\ref{s:3op}, we propose a variant which solves the three
operator inclusion \eqref{eq:zero 3sum}. Finally, in Section~\ref{s:between}, we analyze
a version of the stochastic algorithm which can be considered in between the forward-backward method
 and our proposed method.

\section{Forward-reflected-backward splitting}\label{s:FoRB splitting}
Recall that a set-valued operator $A\colon \Hilbert\setto\Hilbert$ is \emph{monotone} if
$$ \lr{x-y,u-v}\geq 0\quad \forall (x,u),(y,v)\in\gra A, $$
where $\gra A = \{(x,y)\in\Hilbert\times\Hilbert : y\in A(x)\}$
denotes the \emph{graph} of $A$.  A monotone operator is
\emph{maximally monotone} if its graph is not properly contained in
the graph of any other monotone operator. The \emph{resolvent} of a
maximally monotone operator $A\colon \Hilbert\setto\Hilbert$, defined by
$J_A:=(I+A)^{-1}$, is an everywhere single-valued operator
\cite{BC2010}. A single-valued operator $B\colon \Hilbert\to\Hilbert$ is
\emph{$L$-Lipschitz} if $\n{B(x)-B(y)} \leq L\n{x-y}$ for all
$x,y\in\Hilbert$.

In this section, we consider the problem of finding a point
$x\in\Hilbert$ such that
\begin{equation}\label{eq:inclusion soln}
0 \in (A+B)(x),
\end{equation}
where $A\colon \Hilbert\setto\Hilbert$ is maximal monotone, and
$B\colon \Hilbert\to\Hilbert$ is monotone and $L$-Lipschitz. Given initial
points $x_0,x_{-1}\in\Hilbert$, we consider the scheme
\begin{equation}\label{eq:the iteration}
x_{k+1} = J_{\la_k A}\bigl(x_k - \la_k B(x_k) -  \la_{k-1}(B(x_k)-B(x_{k-1})\bigr) \quad\forall k\in\mathbb{N},
\end{equation}
where $(\lambda_k)\subseteq\mathbb{R}_+$ is a sequence
of step-sizes (starting with from $k=-1$). Note that, each iteration of this scheme requires one
forward evaluation and one backward evaluation. Using the definition
of the resolvent $J_{\la_k A}=(I+\lambda_kA)^{-1}$, \eqref{eq:the
    iteration} can be equivalently expressed as the inclusion
\begin{equation}\label{eq:inclusion iteration}
x_{k+1}-x_k + \la_k B(x_k) + \la_{k-1}\left( B(x_k)-B(x_{k-1})\right) \in -\lambda_{k} A(x_{k+1})\quad\forall k\in\mathbb{N}.
\end{equation}

Before turning our attention to the convergence analysis of this
method, we first note some special cases in which it recovers known
methods.
\begin{remark}[Special cases of \eqref{eq:the iteration}]
    We consider three cases in which the proposed algorithm reduces or
    is equivalent to known methods. For simplicity, we only consider
    the fixed step-size case (\emph{i.e.,} $\exists\lambda>0$ such
    that $\lambda_k=\lambda$ for all $k$). In this case,
    \eqref{eq:the iteration} can be expressed compactly as
   \begin{equation}
   \label{eq:simplified iteration}
   x_{k+1} = J_{\la A}\bigl(x_k - 2 \la B(x_k) + \la B(x_{k-1})\bigr) .
   \end{equation}
\begin{enumerate}[label=(\alph*)]
    \item If $B=0$ then \eqref{eq:simplified iteration} simplifies to
    the \emph{proximal point algorithm} \cite{R1976}, that is,
    \eqref{eq:simplified iteration} becomes
	$$ x_{k+1}=J_{\lambda A}(x_k)\quad\forall k\in\mathbb{N}. $$
	\item\label{it:mal B affine} If $A=N_C$ is the normal cone to a set $C$ and $B$ is an affine
        operator then \eqref{eq:simplified iteration} can be expressed
        as
        \begin{equation}\label{eq:maliskys method}
	  x_{k+1} = P_C\bigl(x_k - \la B(2x_k-x_{k-1})\bigr)\quad\forall k\in\mathbb{N},
	  \end{equation}
	which coincides with the \emph{projected reflected gradient method} \cite{M2015} for VIs.
	\item\label{it:mal A=0} If $A=N_{\Hilbert}=0$ then the
        projected reflected gradient method \eqref{eq:maliskys method}
        becomes
        $$ x_{k+1} = x_k - \la B(2x_k-x_{k-1})\quad\forall k\in\mathbb{N}.$$
        Under the change of variables $\overline{x}_k=2x_k-x_{k-1}$,
        this becomes
        $$ \overline{x}_{k+1} = \overline{x}_k - 2\la B(\overline{x}_k) + (x_{k-1}-x_k) =  \overline{x}_k - 2\la B(\overline{x}_k) + \la B(\overline{x}_{k-1})\quad\forall k\in\mathbb{N},$$
        which is precisely \eqref{eq:the iteration} with
        $A=0$. Alternatively, \eqref{eq:simplified iteration} can be
        expressed as the two step recursion
    \begin{equation}
        \label{eq:popov}
        \left\{\begin{aligned}
            y_{k+1} & = y_k - \la B(x_k)\\
            x_{k+1} & = y_{k+1} - \la B(x_k).
        \end{aligned}\right.
    \end{equation}
    This is exactly Popov's algorithm \cite{popov:80} for
    unconstrained VIs.  In this sense, the three methods coincide in
    this case up to a change of variable. Furthermore, in the GANs
    literature, both~\eqref{eq:maliskys method} and~\eqref{eq:popov}
    are also known to be equivalent to the \emph{optimistic gradient}
    method. For details, see the discussion in~\cite{hsieh2019convergence}.
\end{enumerate}
\end{remark}

Before establishing convergence of the method, we require some
preparatory results.

\begin{lemma}\label{lem:weakly convergent sequences}
    Let $(z_k)\subseteq\Hilbert$ be a bounded
    sequence and suppose $\lim_{k\to\infty}\|z_k-z\|$ exists whenever
    $z$ is a cluster point of $(z_k)$. Then
    $(z_k)$ is weakly convergent.
\end{lemma}

Equation~\eqref{abstr_seq} in the following proposition conforms, in
particular, to our proposed method given by
\begin{align*}
x_{k+1} & = J_{\la_k A}\bigl(x_k - \la_k B(x_k) -
\la_{k-1}(B(x_k)-B(x_{k-1})\bigr).
\end{align*}

\begin{proposition}\label{prop:abs}
    Let $F\colon \Hilbert\setto \Hilbert $ be maximally monotone, and let $d_1, v_{2}, u_1$, $v_1, u_{0}\in \Hilbert$ be arbitrary. Define $d_{2}$ as
    \begin{equation}
        \label{abstr_seq}
             d_{2} = J_{F}(d_1-u_1 - (v_1 - u_{0})).
     \end{equation}
     Then, for all $x\in\Hilbert$ and $u \in -F(x)$, we have
\begin{multline}\label{abs_ineq}
            \n{d_{2}-x}^2 + 2\lr{v_{2}-u_1, x-d_{2}} \leq
            \n{d_{1}-x}^2 + 2\lr{v_{1}-u_{0}, x-d_{1}}\\ +
            2\lr{v_1-u_{0},d_1-d_{2}} - \n{d_{1}-d_2}^2 -2\lr{v_{2}-u, d_{2}-x}.
        \end{multline}
    \end{proposition}
\begin{proof}
    By definition of the resolvent,
    $d_1 - u_1 - (v_1-u_{0}) \in d_{2} + F(d_{2})$ and hence, by
    monotonicity of $F$,
    \begin{align*}
           0 &\leq \lr{d_{2}-d_1 + u_1 +  (v_1-u_{0})-u,
               x-d_{2}} \\
      &=\lr{d_{2}-d_1, x-d_{2}} + \lr{u_1-u, x-d_{2}}+\lr{v_1-u_{0},x-d_{2}}.
    \end{align*}
    The first term can expressed as
    \[\lr{d_{2}-d_1, x-d_{2}} = \frac{1}{2}\left( \n{d_1-x}^2 -
            \n{d_{2}-x}^2-\n{d_{2}-d_1}^2\right), \]
and the second and third terms can be rewritten, respectively, as
    \begin{align*}
          \lr{u_1-u, x-d_{2}} & =\lr{v_{2}-u, x-d_{2}} +
          \lr{u_1-v_{2}, x-d_{2}}, \\
          \lr{v_1-u_{0},x-d_{2}}  & =\lr{v_1-u_{0},x-d_1} + \lr{v_1-u_{0},d_1-d_{2}}.
    \end{align*}
The claimed inequality follows by combining these expressions.
\end{proof}
Apart from using the monotonicity of $F$, the proof of Proposition~\ref{prop:abs} only uses simple algebraic manipulations involving $u_1, v_1, u_{0}, v_{2}, d_1, d_{2}$. Nevertheless, the resulting inequality \eqref{abs_ineq} already provides some insight into how our subsequence analysis of \eqref{eq:the iteration intro} proceeds. For instance, the first line of~\eqref{abs_ineq} suggests terms for telescoping so long as the second line can be appropriately estimated.

\begin{lemma}\label{lem:decreasing sequence}
 	Let $x\in(A+B)^{-1}(0)$ and let $(x_k)$ be given by \eqref{eq:the iteration}.
 	Suppose $(\la_k)\subseteq\left[\e,\frac{1-2\e}{2L}\right]$ for some $\e>0$.
 	Then, for all $k\in\mathbb{N}$, we have
 	\begin{multline}
 	\|x_{k+1}-x\|^2+2\la_k\lr{B(x_{k+1})-B(x_k),x-x_{k+1}} + \left(\frac{1}{2}+\e\right)\|x_{k+1}-x_k\|^2 \\
 	\leq \|x_k-x\|^2  + 2\la_{k-1}\lr{B(x_k)-B(x_{k-1}), x-x_{k}}  +\frac{1}{2}\|x_{k}-x_{k-1}\|^2.
 	\end{multline}
    \end{lemma}
    \begin{proof}
By applying Proposition~\ref{prop:abs} with
\begin{equation*}
  \begin{split}
    F&:=\lambda_kA\quad\\
u&:=\la_kB(x) \quad
  \end{split}
\quad
  \begin{split}
       d_1&:=x_k \\
      d_{2}&:=x_{k+1}
\end{split}
\quad
\begin{split}
       u_{0}& :=\la_{k-1}B(x_{k-1}) \quad \\u_1&:=\la_k B(x_k) \quad
\end{split}
\quad
\begin{split}
v_1 &:=\la_{k-1}B(x_k) \\ v_{2} &:= \la_k B(x_{k+1}) ,
\end{split}
\end{equation*}
we obtain the inequality
\begin{multline*}
 \n{x_{k+1}-x}^2 + 2\la_k\lr{B(x_{k+1})-B(x_k),x-x_{k+1}} + \n{x_{k+1}-x_k}^2\\
   \leq \n{x_k-x}^2 + 2\la_{k-1}\lr{B(x_k)-B(x_{k-1}),x-x_{k}} \\
      + 2\la_{k-1}\lr{B(x_k)-B(x_{k-1}),x_k-x_{k+1}} - 2\la_k\lr{B(x_{k+1})-B(x),x_{k+1}-x}.
 \end{multline*}
 Since $B$ is monotone, the last term is nonnegative. Using Lipschitzness of $B$, the second last term can be estimated as
 \begin{equation}\label{eq:lipschitz B for lem:decreasing sequence}
   \begin{split}
       \lr{B(x_k)-B(x_{k-1}),x_k-x_{k+1}}
       & \leq  L\n{x_k-x_{k-1}}\n{x_k-x_{k+1}} \\
       & \leq \frac{L}{2}\left(\n{x_k-x_{k-1}}^2+\n{x_k-x_{k+1}}^2\right).
     \end{split}
 \end{equation}
Thus, altogether, we obtain
 \begin{multline*}
     \n{x_{k+1}-x}^2 + 2\la_k\lr{B(x_{k+1})-B(x_k),x-x_{k+1}} + (1-\la_{k-1} L)\n{x_{k+1}-x_k}^2\\ \leq \n{x_k-x}^2 + 2\la_{k-1}\lr{B(x_k)-B(x_{k-1}),x-x_{k}} + \la_{k-1} L\n{x_k-x_{k-1}}^2.
 \end{multline*}
The claimed inequality follows since $\la_{k-1}L<\frac{1}{2}$ and $1-\la_{k-1}L\geq 1-\frac{1-2\e}{2} = \frac{1}{2}+\e$.
 \end{proof}

We are now ready for the first main result regarding convergence of the proposed method.
\begin{theorem}\label{th:main}
	Let $A\colon \Hilbert\setto\Hilbert$ be maximally monotone, let $B\colon \Hilbert\to\Hilbert$ be monotone and $L$-Lipschitz, and suppose that $(A+B)^{-1}(0)\neq\varnothing$.
	Suppose $(\la_k)\subseteq\left[\e,\frac{1-2\e}{2L}\right]$ for some $\e>0$.
	Given $x_0,x_{-1}\in\Hilbert$, define the sequence $(x_k)$ according to
	   \begin{equation*}
	   x_{k+1} = J_{\la_k A}\bigl(x_k - \la_k B(x_k) - \lambda_{k-1}(B(x_k)-B(x_{k-1})) \bigr) \quad\forall k\in\mathbb{N}.
	   \end{equation*}
	Then $(x_k)$ converges weakly to a point contained in $(A+B)^{-1}(0)$.
\end{theorem}
\begin{proof}
Let $x\in(A+B)^{-1}(0)$. By Lemma~\ref{lem:decreasing sequence}, we have
\begin{multline}\label{eq:key inequality}
	\|x_{k+1}-x\|^2+2\la_k\lr{B(x_{k+1})-B(x_k),x-x_{k+1}} + \left(\frac{1}{2}+\e\right)\|x_{k+1}-x_k\|^2 \\
	\leq \|x_k-x\|^2  + 2\la_{k-1}\lr{B(x_k)-B(x_{k-1}), x-x_{k}}  +\frac{1}{2}\|x_{k}-x_{k-1}\|^2,
	\end{multline}
    which telescopes to yield
    \begin{multline}\label{eq:thm2}
	\|x_{k+1}-x\|^2+2\la_k\lr{B(x_{k+1})-B(x_k),x-x_{k+1}}
	+ \frac{1}{2}\|x_{k+1}-x_k\|^2\\+\e\sum_{i=0}^k\|x_{i+1}-x_i\|^2
	\leq \|x_0-x\|^2 + 2\la_{-1}\lr{B(x_0)-B(x_{-1}), x-x_{0}}+ \frac{1}{2}\n{x_0-x_{-1}}^2.
    \end{multline}
    Using  Lipschitzness of $B$, we can estimate
	\begin{align}\label{eq:lipschitzness of B for th:main}
	\begin{split}
	2\la_k\lr{B(x_{k+1})-B(x_k),x-x_{k+1}}
	&\geq - 2\la_k L\n{x_{k+1}-x_k}\n{x-x_{k+1}} \\
	& \geq -\la_k L \left(\n{x_{k+1}-x_k}^2+\n{x-x_{k+1}}^2\right).
	\end{split}
	\end{align}
	Since $\la_k L \leq (1-2\e)/2<1/2$, substituting the previous equation back into \eqref{eq:thm2} gives
	\begin{multline*}
	\frac{1}{2}\|x_{k+1}-x\|^2+\e\sum_{i=0}^k\|x_{i+1}-x_i\|^2 \\
	\leq \|x_0-x\|^2+2\la_{-1}\lr{B(x_0)-B(x_{-1}), x-x_{0}}+ \frac{1}{2}\n{x_0-x_{-1}}^2,
	\end{multline*}
	from which we deduce that $(x_k)$ is bounded and that $\|x_k-x_{k+1}\|\to 0$.

    Let $\overline{x}$ be a sequential weak cluster point of the bounded sequence $(x_k)$. From \eqref{eq:inclusion iteration},
	\begin{multline}\label{eq:inclusion in proof}
	\frac{1}{\la_{k-1}}\big( x_{k-1}-x_{k} + \la_{k-1}\left( B(x_{k})-B(x_{k-1})\right)\\
	 + \la_{k-2}\left( B(x_{k-2})-B(x_{k-1})\right) \big) \in (A+B)(x_{k})\qquad \forall k\geq 1.
	\end{multline}
	Since $A+B$ is maximally monotone \cite[Corollaries~24.4(i) \& 20.25]{BC2010}, its graph is demiclosed (\emph{i.e.,} sequentially closed in the weak-strong topology on $\Hilbert\times\Hilbert$) \cite[Proposition~20.33]{BC2010}. Thus, by taking the limit along a subsequence of $(x_k)$ which converges to $\overline{x}$ in \eqref{eq:inclusion in proof} and noting that $\lambda_k\geq \e$ for all $k\in\mathbb{N}$, we deduce that $0\in(A+B)(\overline{x})$. To show that $(x_k)$ is weakly convergent, first note that, by combining \eqref{eq:key inequality} and \eqref{eq:lipschitzness of B for th:main}, we deduce existence of the limit
	 \begin{equation}
	 \label{eq:limit}
	 \lim_{k\to\infty}\left(\|x_k-\overline{x}\|^2  + 2\la_{k-1}\lr{B(x_k)-B(x_{k-1}), \overline{x}-x_{k}}  +\frac{1}{2}\|x_{k}-x_{k-1}\|^2\right).
	 \end{equation}
    Since $(x_k)$ and $(\la_k)$ are bounded, $\n{x_k-x_{k+1}}\to 0$, and $B$ is continuous, it then follows that the limit \eqref{eq:limit} is equal to $\lim_{k\to\infty}\|x_{k}-\overline{x}\|^2$.
	Since the cluster point $\overline{x}$ of $(x_k)$ was chosen arbitrarily, the sequence  $(x_k)$ is weakly convergent by Lemma~\ref{lem:weakly convergent sequences} and the proof is complete.
\end{proof}

As an immediate consequence of Theorem~\ref{th:main}, we obtain the
following corollary when the stepsize sequence
$(\lambda_k)$ is constant.
\begin{corollary}
	Let $A\colon \Hilbert\setto\Hilbert$ be maximally monotone, let $B\colon \Hilbert\to\Hilbert$ be monotone and $L$-Lipschitz, and suppose that $(A+B)^{-1}(0)\neq\varnothing$. Choose $\lambda\in \left(0,\frac{1}{2L}\right)$.
	Given $x_0,x_{-1}\in\Hilbert$, define the sequence $(x_k)$ according to
	\begin{equation*}
 	x_{k+1} = J_{\la A}\bigl(x_k - 2\la B(x_k) +\la B(x_{k-1})\bigr) \quad\forall k\in\mathbb{N}.
	\end{equation*}
	Then $(x_k)$ converges weakly to a point contained in $(A+B)^{-1}(0)$.
\end{corollary}

\begin{remark}\label{rem:metric}
In practice, it can be desirable to analyze an algorithm with respect to an auxiliary metric to encourage faster convergence. This is done by considering the metric induced by the inner product $\lr{\cdot,\cdot}_M$ corresponding to a symmetric positive definite operator $M\colon \Hilbert \to \Hilbert$. In the case of saddle point problems, for example, choosing the operator $M$ to be a diagonal scaling matrix gives different weights to primal and dual variables. To keep our presentation as simple and as clear as possible, we present our analysis only for the case when $M = I$. Nevertheless, the more general case can be easily obtained through a straightforward modification of the proof. In particular, instead of \eqref{eq:the iteration}, we can consider the iteration
\begin{equation*} 
		x_{k+1} = J_{\la_k A}^M \left( x_k - M^{-1} \left[\la_k B(x_k) +
		\la_{k-1} (B(x_k) - B(x_{k-1}))\right] \right) \quad\forall k\in\mathbb{N},
\end{equation*}
where $J_A^M = (I + M^{-1}A)^{-1}$ denotes the \emph{generalized resolvent} of $A$.
\end{remark}

\begin{remark}
    Since the main focus of this work lies in the development and
    analysis of new methods, we delay a more thorough computation
    comparison for future investigation. Nevertheless, the following
    example provides a specific problem for which the
    forward-reflected-backward method is faster than Tseng's method. We make
    no claims about the performance of the proposed method in general.

Consider \eqref{eq:zero 2sum} with
$\Hilbert=\mathbb{R}^n\times\mathbb{R}^n$, $A(z_1,z_2)=(0,0)$ and
$B(z_1,z_2)=(z_2,-z_1)$. Note that zero is the unique solution to this problem
and that the operator $B$ is $1$-Lipschitz. Let us also denote the
identity operators on $\Hilbert$ and $\mathbb{R}^n$ by $I_{\Hilbert}$
and $I_n$, respectively. This is a classical example of a monotone
inclusion, where the forward-backward method fails.

\paragraph{Tseng's method} By eliminating $y_k$ from \eqref{eq:tseng}
and using the identity $B^2=-I_{\Hilbert}$, Tseng's method can be expressed as
  $$ x_{k+1} = T(x_k)=T^{k+1}(x_0)\text{~where~} T:= (1-\lambda^2)I_\Hilbert-\lambda B. $$
Since $\lr{x_k,B(x_k)}=0$ and $\n{x_k}=\n{B(x_k)}$, we have
  $$ \|x_{k+1}\|^2 = \|T(x_k)\|^2 = \bigl( (1-\la^2)^2 +\la^2\bigr) \|x_k\|^2. $$
Let $\la\in(0,1)$. The sequence $(x_k)$ therefore converges $Q$-linearly to zero with rate
  $$ \rho:=  \sqrt{(1-\la^2)^2+\la^2}=\sqrt{1-\la^2+\la^4}<1. $$
In fact, this shows the optimal stepsize is $\la=1/\sqrt{2}$ which
gives a rate of $\sqrt{3}/2$. (Note that the optimal rate does not occur
for the largest possible stepsize).

\paragraph{Forward-reflected-backward splitting} The forward-reflected-backward  method with constant stepsize $\la\in(0,1/2)$ can be expressed as
\begin{equation*}
\binom{x_{k+1}}{x_k} = T\binom{x_{k}}{x_{k-1}}=T^{k+1}\binom{x_{0}}{x_{-1}}\text{~~where~~}T:=\begin{bmatrix}I_\Hilbert-2\la B & \la B \\ I_\Hilbert & 0 \\ \end{bmatrix}
\end{equation*}
The eigenvalues of $T$ are given by
$ \frac{1}{2}\pm\frac{1}{2}i\,\sqrt {8\,{\la}^{2}-1-4i\,\la\sqrt {1-4\,{\la}^{2}}}$.
By choosing the stepsize $\la\approx1/2$, we deduce that $(x_k)$ converges $R$-linearly with a rate that be made arbitrarily close to $\left|\frac{1}{2}\pm \frac{1}{2}i\right|=\frac{1}{\sqrt{2}}. $

Since $1/\sqrt{2}<\sqrt{3}/{2}$, we conclude that the forward-reflected-backward method is faster than
Tseng's method for this particular problem. Note that this comparison is in terms of the number of iterations.
\end{remark}

\subsection{Linear Convergence}\label{s:linear convergence}
In this section, we establish $R$-linear convergence of the sequence generated by the forward-reflected-backward method when $A$ is \emph{strongly monotone}. Recall that $A\colon \Hilbert\setto\Hilbert$ is \emph{$m$-strongly monotone} if $m>0$ and
  $$ \lr{x-y,u-v} \geq m\n{x-y}^2\quad\forall (x,u),(y,v)\in\gra A. $$
  Strong monotonicity is a standard assumption for proving linear
  convergence of first order methods. We also note that there is no
  loss of generality in assuming that $A$ is strongly monotone. For if
  $B$ is $m$-strongly monotone, we can always augment the operators by
  the identity, \emph{i.e.,} $A+B=(A+mI)+(B-mI)$, without destroying
  monotonicity and Lipschitz continuity. Notice this  does not
  complicate computing the resolvent of $(A + mI)$, as
  we have $J_{A + mI}(x) = J_{\frac{A}{1+m}}\bigl(\frac{x}{1+m}\bigr)$ for all
  $x\in \Hilbert$.

\begin{theorem}
	Let $A\colon \Hilbert\setto\Hilbert$ be maximally monotone and $m$-strongly monotone, let $B\colon \Hilbert\to\Hilbert$ be monotone and $L$-Lipschitz, and suppose $(A+B)^{-1}(0)\neq\varnothing$. Let $\la\in\left(0,\frac{1}{2L}\right)$. Given $x_0,x_{-1}\in\Hilbert$, define the sequence $(x_k)$ according to
	\begin{equation*}
	x_{k+1} = J_{\la A}\bigl( x_k - 2\la B(x_k) +\la B(x_{k-1}) \bigr) \quad\forall k\in\mathbb{N}.
	\end{equation*}
	Then $(x_k)$ converges $R$-linearly to the unique element of $(A+B)^{-1}(0)$.
\end{theorem}
\begin{proof}
	Let $x\in(A+B)^{-1}(0)$. Using strong monotonicity of $A$ (in
        place of monotonicity) in Proposition~\ref{prop:abs} and
        propagating the resulting inequality through the proof of Lemma~\ref{lem:decreasing sequence} gives the inequality
	\begin{multline*}
	\left(1+2m\la \right)\n{x_{k+1}-x}^2 + 2\la\lr{B(x_{k+1})-B(x_k),x-x_{k+1}} + \left(1-\la L\right)\n{x_{k+1}-x_k}^2  \\
	\leq \n{x_k-x}^2
	+ 2\la\lr{B(x_k)-B(x_{k-1}),x-x_{k}} + \frac{1}{2}\n{x_k-x_{k-1}}^2.
	\end{multline*}
    By denoting $\varepsilon:=\min\{\frac12-\la L,5m\la\}>0$, this inequality implies
	 \begin{equation}\label{eq:R linear recursion}
	  (1+4m\la)a_{k+1}+ b_{k+1} +\varepsilon\|x_{k+1}-x_k\|^2\leq a_k+\,b_k,
	 \end{equation}
	where the nonnegative sequences $(a_k)$ and $(b_k)$ are given by
	\begin{align*}
	a_k  \coloneqq\, &\frac{1}{2}\n{x_k-x}^2 \geq 0,\\
	b_k \coloneqq \, &\frac{1}{2} \n{x_k-x}^2
	+ 2\la\lr{B(x_k)-B(x_{k-1}),x-x_{k}} +
              \frac{1}{2}\n{x_k-x_{k-1}}^2 \\  \geq\, & \frac{1}{2} \n{x_k-x}^2 - 2\lambda L\n{x_k-x_{k-1}}\n{x_{k}-x} + \frac{1}{2}\n{x_k-x_{k-1}}^2\geq0.
	\end{align*}
	Using Lipschitzness of $B$, we have
	\begin{align}\label{eq:R linear under}
	 &(1+4m\la)a_{k+1}+ b_{k+1} +\varepsilon\|x_{k+1}-x_k\|^2 \notag \\
          &\quad= \left(1+4m\la-\frac{\varepsilon}{2}\right)a_{k+1}+
            \left(1+\frac{\varepsilon}{2}\right)b_{k+1}+\frac{3\varepsilon}{4}\|x_{k+1}-x_k\|^2
            - \varepsilon\la \lr{B(x_{k+1})-B(x_{k}),x-x_{k+1}}\notag \\
	 &\quad\geq \left(1+4m\la-\frac{\varepsilon}{2}\right)a_{k+1}+
         \left(1+\frac{\varepsilon}{2}\right)b_{k+1}+\frac{3\varepsilon}{4}\|x_{k+1}-x_k\|^2
         - \varepsilon\la L\|x_{k+1}-x_{k}\|\|x_{k+1}-x\| \notag\\
	 &\quad\geq \left(1+4m\la-\frac{3\varepsilon}{4}\right)a_{k+1}+ \left(1+\frac{\varepsilon}{2}\right)b_{k+1}+\frac{\varepsilon}{2}\|x_{k+1}-x_k\|^2.
	\end{align}
    Denote
    $\alpha\coloneqq\min\{1+4m\la-3\varepsilon/4,1+\varepsilon/2\}>1$,
    which is true due to  $\e\leq 5m\la$. Combining \eqref{eq:R linear recursion} and \eqref{eq:R linear under} yields $\alpha(a_{k+1}+b_{k+1}) \leq a_k+b_k$. Iterating this inequality gives
        $$ a_{k+1} \leq a_{k+1}+b_{k+1} \leq \frac{1}{\alpha}(a_k+b_k) \leq \dots \leq \frac{1}{\alpha^{k+1}}(a_0+b_0), $$
     which establishes that $x_k\to x$ with $R$-linear rate.	Since $x$ was chosen arbitrarily from $(A+B)^{-1}(0)$, it must be unique.
\end{proof}

\section{Forward-reflected-backward splitting with linesearch}\label{s:linesearch}
The algorithm presented in the previous section required information about the single-valued operator's Lipschitz constant in order to select an appropriate stepsize. In practice, this requirement is  undesirable for several reasons. Firstly, obtaining the Lipschitz constant (or an estimate) is usually non-trivial and often a computationally expensive problem itself. Secondly, as a global constant, the (global) Lipschitz constant can often lead to over-conservative stepsizes although local properties (around the current iterate) may permit the use of larger stepsizes and ultimately lead to faster convergence. Finally, when the single-valued operator is not Lipschitz continuous, any fixed stepsize scheme based on Lipschitz continuity will potentially fail to converge.

To address these shortcomings, most known methods can incorporate an additional procedure called
\emph{linesearch} (or \emph{backtracking}) which is run in each iteration. It is worth noting however, that in the more restrictive context of variational inequalities, the method proposed in \cite{malitsky2018golden} overcomes the aforementioned difficulties without resorting to a linesearch procedure.

In what follows, we show that the forward-reflected-backward method with such a linesearch procedure converges whenever the single-valued operator is \emph{locally Lipschitz}.

\begin{algorithm}[t]
    \caption{The forward-reflected-backward method with linesearch.\label{alg:ls}}
    {\bfseries Initialization:} Choose $x_0,x_{-1} \in\Hilbert$, $\la_0,\la_{-1}>0$, $\delta \in (0,1)$, and $\sigma \in (0,1)$.

    {\bfseries Iteration:} Having $x_k$, $\la_{k-1}$, and $B(x_{k-1})$, choose $\rho\in
    \{1,\sigma^{-1}\}$ and compute
    \begin{equation}
        \label{ls:xk+1}
        x_{k+1} := J_{\la_k A}\bigl(x_k - \la_k B(x_k) -
                \la_{k-1}(B(x_k)-B(x_{k-1}))\bigr),
            \end{equation}
    where $\la_k = \rho \la_{k-1}\sigma^i$ with $i$ being the smallest nonnegative integer satisfying
            \begin{equation}
                \label{ls:break}
                \la_k \n{B(x_{k+1})-B(x_k)} \leq
\frac{\delta}{2}\n{x_{k+1}-x_k}.
            \end{equation}
\end{algorithm}

\begin{remark}
	The parameter $\rho$ in Algorithm~\ref{alg:ls} has been introduced to allow for greater flexibility in the choice of possible stepsizes.
	Indeed, there are two possible scenarios for the value of $\la_k$ in the first iteration of the linesearch procedure (\emph{i.e.,} when $i=0$): either $\rho=\sigma^{-1}$ and $\la_k = \sigma^{-1}\la_{k-1}> \la_{k-1}$, or $\rho=1$ and $\la_k = \la_{k-1}$. The former, more aggressive scenario allows for the possibility of larger stepsizes at the price of a potential increase in the number of linesearch iterations.
\end{remark}

The following lemma shows that the  linesearch procedure described in Algorithm~\ref{alg:ls} is well-defined so long as the operator $B$ is locally Lipschitz continuous.
\begin{lemma}\label{lem:linesearch-stop}
	Suppose $B\colon \Hilbert\to\Hilbert$ is locally Lipschitz. Then the linesearch procedure in \eqref{ls:xk+1}--\eqref{ls:break} always terminates. i.e., $(\la_k)$ is well defined.
\end{lemma}
\begin{proof}
	Denote $x_{k+1}(\la) := J_{\la A}(x_k - \la B(x_k) - \la_{k-1}(B(x_k)-B(x_{k-1})))$. From \cite[Theorem~23.47]{BC2010}, we have that $J_{\lambda A}(x_{k+1}(0))\to P_{\overline{\dom A}}(x_{k+1}(0))$ as $\la\searrow 0$ which, together with the nonexpansivity of $J_{\lambda A}$, yields
	\begin{align*}
	 &\n{x_{k+1}(\lambda)-P_{\overline{\dom A}}\,x_{k+1}(0) } \\
	 &\qquad \leq \n{x_{k+1}(\lambda)-J_{\lambda A}(x_{k+1}(0)) } + \n{J_{\lambda A}(x_{k+1}(0))-P_{\overline{\dom A}}(x_{k+1}(0)) } \\
	 &\qquad \leq \lambda\|B(x_k)\| + \n{J_{\lambda A}(x_{k+1}(0))-P_{\overline{\dom A}}(x_{k+1}(0)) }.
	\end{align*}
    By taking the limit as $\lambda\searrow0$, we deduce that $x_{k+1}(\lambda)\to P_{\overline{\dom A}}(x_{k+1}(0))$.

	Now, by way of a contradiction, suppose that the linesearch procedure in Algorithm~\ref{alg:ls} fails to terminate at the $k$-th iteration.
	Then, for all $\la =\rho \la_{k-1}\sigma^i$ with $i=0,1,\dots$, we have
	\begin{equation}\label{ls:lemma-terminate-1}
	\rho \la_{k-1}\sigma^i \n{B(x_{k+1}(\la))-B(x_k)} > \frac{\delta}{2} \n{x_{k+1}(\la)-x_k}.
	\end{equation}
    On one hand, taking the limit as $i\to\infty$ in \eqref{ls:xk+1} gives $P_{\overline{\dom A}}(x_{k+1}(0))=x_k$. On the other hand, since $B$ is locally Lipschitz at $x_k$ there exists $L>0$ such that for $i$ sufficiently large, we have
    $$ \rho \la_{k-1}\sigma^i \n{B(x_{k+1}(\la))-B(x_k)} > \frac{\delta}{2} \n{x_{k+1}(\la)-x_k} \geq \frac{\delta L}{2}\n{B(x_{k+1}(\la))-B(x_k)}. $$
    Dividing both sides by $\n{B(x_{k+1}(\la))-B(x_k)}$ gives $\delta L/2 < \rho\lambda_{k-1}\sigma^i$. Since $\sigma^i\to 0$ as $i\to\infty$, this inequality gives a contradiction which completes the proof.
\end{proof}

The next lemma is a direct extension of Lemma~\ref{lem:decreasing sequence}.
\begin{lemma}\label{lem:decreasing sequence-ls}
    Let $x\in(A+B)^{-1}(0)$ and let $(x_k)$ be
    generated by Algorithm~\ref{alg:ls}. Then there exists $\e > 0$ such that, for all
    $k\in\mathbb{N}$, we have
	\begin{multline}
		\n{x_{k+1} - x}^2 +
                2\la_k\lr{B(x_{k+1})-B(x_k),x-x_{k+1}} + \left(\frac 1
                    2 + \e \right)\n{x_{k+1}-x_k}^2  \\
		\leq \n{x_k-x}^2
		+ 2\la_{k-1}\lr{B(x_k)-B(x_{k-1}),x-x_{k}} + \frac 1 2
                \n{x_k-x_{k-1}}^2.
	\end{multline}
\end{lemma}
\begin{proof}
   The proof is exactly the same as Lemma~\ref{lem:decreasing
   	sequence} with the only change being that instead of using Lipschitzness of $B$ to deduce the inequality \eqref{eq:lipschitz B for lem:decreasing sequence}, we use \eqref{ls:break}, which is well-defined due to Lemma~\ref{lem:linesearch-stop}.
\end{proof}

\begin{theorem}\label{th:linesearch}
    Let $\Hilbert$ be finite dimensional,  $A\colon \Hilbert\setto\Hilbert$ be maximally monotone, and
    $B\colon \Hilbert\to\Hilbert$ be monotone and locally Lipschitz continuous, and suppose
    that $(A+B)^{-1}(0)\neq\varnothing$. Then the sequence
    $(x_k)$ generated by Algorithm~\ref{alg:ls}
	converges to a point contained in $(A+B)^{-1}(0)$.
\end{theorem}
\begin{proof}
    We argue similarly to Theorem~\ref{th:main} but using Lemma~\ref{lem:decreasing sequence-ls} in place of Lemma~\ref{lem:decreasing sequence}, and \eqref{ls:break} in place of Lipschitzness of $B$. This yields  \eqref{eq:lipschitzness of B for th:main} from which we deduce that $(x_k)$ is bounded and $\|x_k-x_{k+1}\|\to 0$. As a locally Lipschitz operator on finite dimensional space, $B$ is Lipschitz on bounded sets. Thus, since $(x_k)$ is bounded, there exists a constant $L>0$ such that
      \begin{equation}
      \label{eq:lip on bounded set}
      \|B(x_{k+1})-B(x_k)\|\leq L \|x_{k+1}-x_k\|\quad\forall k\in\mathbb{N}.
      \end{equation}
   By combining \eqref{ls:break} and \eqref{eq:lip on bounded set}, we see that $(\lambda_k)$ is bounded away from zero. The remainder of the proof is the same as Theorem~\ref{th:main}.
\end{proof}

\section{Relaxed inertial forward-reflected-backward splitting}\label{s:relaxed inertial}
In this section, we consider a relaxed inertial variant of the forward-reflected-backward splitting algorithm. Such variants are of interest in practice because they have the potential to improve performance as well as the range of admissible stepsizes. A treatment of a relaxed inertial variant of the forward-backward method with $B$ cocoercive and its relation to Nesterov-type acceleration techniques can be found in \cite{attouch2018convergence}.

Consider the monotone inclusion
  $$ \text{find~}x\in\Hilbert\text{ such that }0\in (A+B)(x), $$
where $A\colon \Hilbert\setto\Hilbert$ is maximally monotone, and $B\colon \Hilbert\to\Hilbert$ is either monotone and $L$-Lipschitz continuous or $1/L$-cocoercive.  The relaxed inertial algorithm is given by
\begin{equation}\label{eq:inertia_scheme_first}
\left\{\begin{aligned}
 z_{k+1} &:= J_{\la A}\bigl( x_k - \la B(x_k) - \frac{\la}{\b }(B(x_k)- B(x_{k-1})) + \frac{\a }{\b }(x_k-x_{k-1}) \bigr) \\
 x_{k+1} &:= (1-\b)  x_k + \b z_{k+1}
\end{aligned}\right.
\end{equation}
for all $k\in \N$ and for appropriately chosen parameters $\a\geq 0$
and $\b,\la>0$ whose precise form depends on the properties of
$B$. By denoting $B':=B-\frac{\a }{\la}I$, the scheme can be expressed
as
\begin{equation}\label{eq:inertia_scheme}
\left\{\begin{aligned}
z_{k+1} &:=
J_{\la A}\bigl( x_k - \la B(x_k) - \frac{\la}{\b }(B'(x_k)- B'(x_{k-1})) \bigr) \\
x_{k+1} &:= (1-\b) x_k + \b z_{k+1}
\end{aligned}\right.\qquad\forall k\in\mathbb{N}.
\end{equation}

To prove convergence of this scheme, we first prove two lemmas.
\begin{lemma}\label{lem:B lip and coco}
	Suppose $B\colon \Hilbert\to\Hilbert$ is monotone and $\rho\geq 0$. Then the operator
	$B':=B-\rho I$ is $L'$-Lipschitz with $L'$  given by
	\begin{equation}
	\label{eq:L'}
	L':=\begin{cases}
	L + \rho &\text{if $B$ is $L$-Lipschitz}, \\
	L-\rho            &\text{if $B$ is $1/L$-cocoercive and $\rho\leq\frac{L}{2}$}, \\
	\rho              &\text{if $B$ is $1/L$-cocoercive and $\rho>\frac{L}{2}$}.
	\end{cases}
	\end{equation}
\end{lemma}
\begin{proof}
	Let $x,y\in\Hilbert$. When $B$ is $L$-Lipschitz, we have
	\begin{equation*}
	\n{(B-\rho I)x-(B-\rho I)y} \leq \n{Bx-By} + \rho \n{x-y}
                                          \leq (L+\rho)\n{x-y},
	\end{equation*}
	which establishes the first case. For the second and third cases, first observe that $1/L$-cocoercivity of $B$ yields
	\begin{equation*}
	\begin{aligned}
	\n{(B-\rho I)x-(B-\rho I)y}^2
	&= \n{Bx-By}^2-2\rho\lr{Bx-By,x-y} +\rho^2\n{x-y}^2 \\
	&\leq \left(1 -\frac{2\rho}{L} \right)\n{Bx-By}^2+\rho^2\n{x-y}^2.
	\end{aligned}
	\end{equation*}
	On one hand, if $\rho > \frac{L}{2}$, then $1-\frac{2\rho}{L}<0$ and $\rho$-Lipschitzness of $B'$ follows. On the other hand, if $\rho \leq \frac{L}{2}$, then
	\begin{align*}
	\n{(B-\rho I)x-(B-\rho I)y}^2
	& \leq  \left(L^2 -\frac{2\rho}{L} L^2 +\rho^2 \right)\n{x-y}^2  \\
	&= (L-\rho)^2\n{x-y}^2,
	\end{align*}
	which shows that $B'$ is $(L-\rho)$-Lipschitz. The proof is now complete.
\end{proof}
Note that it is possible to slightly improve the estimate of $L'$ in the case when $B$ is $L$-Lipschitz and monotone to $L'=\sqrt{L^2+\rho^2}$.
However, the benefits of using a new bound are minimal since, in our case $\rho $, will take small values relative to $L$.

In the following lemma, we use the following form of \eqref{abs_ineq} from Proposition~\ref{prop:abs}.
\begin{multline}\label{abs_ineq2}
\n{d_{2}-x}^2 + 2\lr{u-u_1, x-d_{2}} \leq
\n{d_{1}-x}^2 + 2\lr{v_{1}-u_{0}, x-d_{1}}\\ +
2\lr{v_1-u_{0},d_1-d_{2}} - \n{d_{1}-d_2}^2.
\end{multline}

\begin{lemma}
  Let $x\in (A+B)^{-1}(0)$, let $(x_k)$ be given by \eqref{eq:inertia_scheme} and consider constants $\alpha\geq0$ and $\beta,\la>0$. Then
\begin{multline*}
(1-\a )\n{x_{k+1}-x}^2 +2\la\lr{B'(x_{k+1})-B'(x_k),x-x_{k+1}} + b_{k+1} \\
\leq (1-\a )\|x_k-x\|^2 + 2\la\lr{B'(x_k)-B'(x_{k-1}),x-x_k}+  b_k\\
+\frac{\la L'}{\b}\n{x_k-x_{k-1}}^2 -\left(\frac{2-\b-\la  L'}{\b}-\a \right) \n{x_{k+1}-x_k}^2.
\end{multline*}
where $b_k:=2\la\lr{B(x)-B(x_k),x-x_{k}}\geq0$ and $L'$ is the Lipschitz constant of $B'$.
\end{lemma}
\begin{proof}
Let $u:=\la B(x)$. As $u\in -\la A(x)$, applying \eqref{abs_ineq2} with
\begin{equation*}
  \begin{split}
    F&:=\lambda A\quad\\
     u_1&:=\la B(x_k) \quad \\
  \end{split}
\qquad
  \begin{split}
       d_1&:=x_k \\
      d_{2}&:=z_{k+1}
\end{split}
\qquad
\begin{split}
 u_{0}& :=(\la / \b)  B'(x_{k-1}) \quad \\ v_1 &:=(\la/\b) B'(x_k)
\end{split}
\end{equation*}
yields the inequality
\begin{multline}\label{inert:est1}
\|z_{k+1}-x\|^2 + 2\la\lr{B(x)-B(x_k),x-z_{k+1}} \leq \|x_k-x\|^2 - \n{x_k-z_{k+1}}^2\\ +  \frac{2\la}{\b }\lr{B'(x_k)-B'(x_{k-1}),x-x_k} +\frac{2\la}{\b }\lr{B'(x_k)-B'(x_{k-1}),x_k-z_{k+1}} .
\end{multline}
Using the identity $z_{k+1}=\bigl(x_{k+1}-(1-\b)x_k\bigr)/\b$ and the definition of $B'$, the second term in \eqref{inert:est1} can be expressed as
\begin{align*}
& 2\la\lr{B(x)-B(x_k),x-z_{k+1}} + \frac{1-\b}{\b}b_k - \frac{1}{\b}b_{k+1}\\
&= \frac{2\la}{\b }\lr{B'(x_{k+1})-B'(x_k),x-x_{k+1}} +\frac{2\a}{\b}\lr{x_{k+1}-x_k,x-x_{k+1}} \\
&= \frac{2\la}{\b}\lr{B'(x_{k+1})-B'(x_k),x-x_{k+1}} + \frac{\a}{\b}\bigl(\n{x_k-x}^2-\n{x_{k+1}-x_k}^2-\n{x_{k+1}-x}^2\bigr)
\end{align*}
Substituting this back into \eqref{inert:est1} and using  $z_{k+1}-x_k=(x_{k+1}-x_k)/\b$ gives
\begin{multline}\label{inert:est2}
\|z_{k+1}-x\|^2-\frac{\a }{\b}\n{x_{k+1}-x}^2 +\frac{2\la}{\b}\lr{B'(x_{k+1})-B'(x_k),x-x_{k+1}} + \frac{1}{\b}b_{k+1}\\
 \leq \left(1-\frac{\a }{\b}\right)\|x_k-x\|^2 + \frac{2\la}{\b }\lr{B'(x_k)-B'(x_{k-1}),x-x_k}+ \frac{1-\b}{\b}b_k \\
+\frac{2\la}{\b }\lr{B'(x_k)-B'(x_{k-1}),x_k-z_{k+1}} +\left(\frac{\a }{\b}-\frac{1}{\b^2}\right)\n{x_{k+1}-x_k}^2.
\end{multline}
By using the identity
\begin{equation*}
\n{z_{k+1}-x}^2 = \frac{1}{\b}\n{x_{k+1}-x}^2 -
\frac{1-\b}{\b}\n{x_k-x}^2 + \frac{1-\b}{\b^2}\n{x_{k+1}-x_k}^2
\end{equation*}
in \eqref{inert:est2} and multiplying both sides by $\b$, we obtain
\begin{multline}\label{inert:est3}
(1-\a )\n{x_{k+1}-x}^2 +2\la \lr{B'(x_{k+1})-B'(x_k),x-x_{k+1}}
+ b_{k+1}\\
\leq (1-\a ) \|x_k-x\|^2 + 2\la \lr{B'(x_k)-B'(x_{k-1}),x-x_k}+ (1-\b)b_k \\
+2\la\lr{B'(x_k)-B'(x_{k-1}),x_k-z_{k+1}}-\left(\frac{2-\b}{\b}-\a  \right)\n{x_{k+1}-x_k}^2  .
\end{multline}
Since $B'$ is $L'$-Lipschitz, the second last term can be estimated by
\begin{align*}
2\la\lr{B'(x_k)-B'(x_{k-1}),x_k-z_{k+1}}
&= \frac{2\la}{\b}\lr{B'(x_k)-B'(x_{k-1}),x_k-x_{k+1}} \\
&\leq \frac{\la L'}{\b}\left( \n{x_k-x_{k-1}}^2 + \n{x_{k+1}-x_k}^2 \right).
\end{align*}
The claimed inequality follows by substituting this estimate back into \eqref{inert:est3}.
\end{proof}

\begin{theorem}\label{th:relaxed inertial}
    Let $A\colon \Hilbert\setto\Hilbert$ be maximally monotone and let
    $B\colon \Hilbert\to\Hilbert$ be monotone with $(A+B)^{-1}(0)\neq\varnothing$.
    Suppose $\a\in[0,1)$, $\beta\in(0,1]$, $\la>0$ and either
    \begin{enumerate}[label=(\alph*)]
    	\item\label{it:a} $B$ is $L$-Lipschitz and
      \begin{equation}
      \label{ineq:a}
      \la < \min\left\{\frac{2-\beta -\a \b -2\a}{2L}, \frac{1-\a-\a
              \b}{\b L}\right\}, \quad \text{or}
  \end{equation}
    	\item\label{it:b} $B$ is $(1/L)$-cocoercive, $\alpha<\frac{2-\beta}{2+\beta}$ and
          \begin{equation}
          \label{ineq:b}
          \la <  \min\left\{ \frac{2-\b -\a \b + 2\a}{2L},\frac{1-\a+\a\b}{\b L}\right\}.
          \end{equation}
    \end{enumerate}
    Given $x_0,x_{-1}\in\Hilbert$, define the sequences
    $(x_k)$ and $(z_k)$ according to \eqref{eq:inertia_scheme_first}
    Then $(x_k)$ converges weakly to a point in $(A+B)^{-1}(0)$.
\end{theorem}
\begin{figure}[t]
    \begin{subfigure}[b]{0.49\linewidth}
        \centering
		\begin{tikzpicture}[scale=0.75]
		\definecolor{color0}{rgb}{0.917647058823529,0.917647058823529,0.949019607843137}
		\begin{axis}[
		title={Monotone case},
		xlabel={$\beta$},
		ylabel={$\alpha$},
		xmin=0, xmax=1,
		ymin=0, ymax=1,
		tick align=outside,
		xmajorgrids,
		x grid style={white},
		ymajorgrids,
		y grid style={white},
		axis line style={white},
		axis background/.style={fill=color0}
		]
		\addplot graphics [includegraphics cmd=\pgfimage,xmin=0, xmax=1, ymin=0, ymax=1] {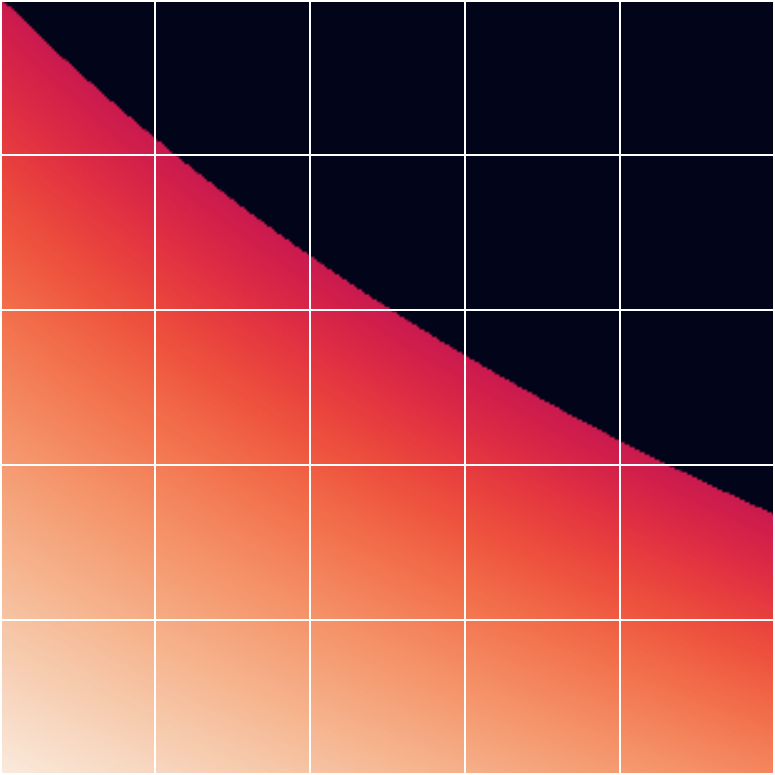};
		\end{axis}
		\end{tikzpicture}
    \end{subfigure}
    \begin{subfigure}[b]{0.49\textwidth}
        \centering
		\begin{tikzpicture}[scale=0.75]
		\definecolor{color0}{rgb}{0.917647058823529,0.917647058823529,0.949019607843137}
		\begin{axis}[
		title={Cocoercive case},
		xlabel={$\beta$},
		ylabel={$\alpha$},
		xmin=0, xmax=1,
		ymin=0, ymax=1,
		tick align=outside,
		xmajorgrids,
		x grid style={white},
		ymajorgrids,
		y grid style={white},
		axis line style={white},
		axis background/.style={fill=color0}
		]
		\addplot graphics [includegraphics cmd=\pgfimage,xmin=0, xmax=1, ymin=0, ymax=1] {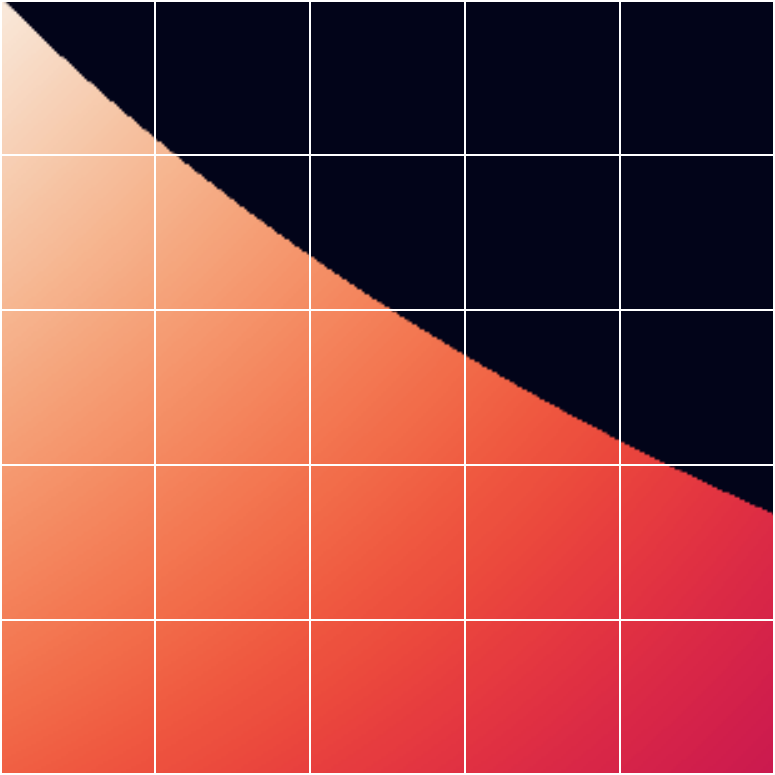};
		\end{axis}
		\end{tikzpicture}
        \end{subfigure}
        \caption{The upper bound for admissible values for $\la L$ as a function of $\a$ and $\b$ according to \eqref{ineq:a} and \eqref{ineq:b}.
        	    Black regions denote infeasible combinations. Lighter colors indicate a higher admissible value.}
        \label{fig:param}
\end{figure}

\begin{proof}
Let $L'$ denote the Lipschitz constant of $B':=B-\frac{\a}{\la}I$. By combining Lemma~\ref{lem:B lip and coco} with the assumptions in each of the respective cases, we obtain
\begin{equation}\label{eq:eps}
  \e:=\min\left\{ \frac{2-\b(1+\a)}{2},\frac{1-\a}{\b}\right\}-\la L'>0.
\end{equation}
Then Lemma~\ref{eq:inertia_scheme} together with \eqref{eq:eps} gives
\begin{multline*}
(1-\a )\n{x_{k+1}-x}^2 +2\la\lr{B'(x_{k+1})-B'(x_k),x-x_{k+1}} + b_{k+1} \\
\leq (1-\a )\|x_k-x\|^2 + 2\la\lr{B'(x_k)-B'(x_{k-1}),x-x_k}+  b_k\\
+\frac{\la L'}{\b}\n{x_k-x_{k-1}}^2 -\left(\frac{\la L'}{\b}+\e \right) \n{x_{k+1}-x_k}^2,
\end{multline*}
which telescopes to yield
\begin{multline}\label{eq:telescopedmf!}
(1-\a )\n{x_{k+1}-x}^2 +2\la\lr{B'(x_{k+1})-B'(x_k),x-x_{k+1}} + b_{k+1} \\
\leq (1-\a )\|x_0-x\|^2 + 2\la\lr{B'(x_0)-B'(x_{-1}),x-x_0}+  b_0\\
+\frac{\la L'}{\b}\n{x_0-x_{-1}}^2 -\frac{\la L'}{\b}\n{x_{k+1}-x_k}^2 - \e \sum_{i=1}^k\n{x_{i+1}-x_i}^2.
\end{multline}
The $L'$-Lipschitz continuity of $B'$ together with \eqref{eq:eps} gives
\begin{align*}
2\la\lr{B'(x_{k+1})-B'(x_k),x-x_{k+1}}
  &\geq -2\la L'\n{x_{k+1}-x_k}\n{x_{k+1}-x} \\
  &\geq -\frac{\la L'}{\beta}\n{x_{k+1}-x_k}^2 - \beta \la L'\n{x_{k+1}-x}^2 \\
  &\geq -\frac{\la L'}{\beta}\n{x_{k+1}-x_k}^2 - (1-\alpha-\b\e)\n{x_{k+1}-x}^2.
\end{align*}
This, together with \eqref{eq:telescopedmf!} and the fact that $b_{k+1}\geq 0$, yields the inequality
\begin{multline*}
\b\e\n{x_{k+1}-x}^2 + \e \sum_{i=1}^k\n{x_{i+1}-x_i}^2 \\
\leq (1-\a )\|x_0-x\|^2 + 2\la\lr{B'(x_0)-B'(x_{-1}),x-x_0}+  b_0
+\frac{\la L'}{\b}\n{x_0-x_{-1}}^2,
\end{multline*}
which shows that $(x_k)$ is bounded and $\n{x_k-x_{k+1}}\to0$. The remainder of the proof follows a similar argument to Theorem~\ref{th:main}.
\end{proof}
The admissible values of $\a$ and $\b$ for the two cases in Theorem~\ref{th:relaxed inertial} are shown in Figure~\ref{fig:param}.
By setting $\beta=1$ in Theorem~\ref{th:relaxed inertial}, we obtain the following inertial algorithm.
\begin{corollary}\label{cor:inertial}
	Let $A\colon \Hilbert\setto\Hilbert$ be maximally monotone and let
	$B\colon \Hilbert\to\Hilbert$ be monotone with $(A+B)^{-1}(0)\neq\varnothing$.
	Suppose $\a\in[0,1/3)$, $\la>0$, and either
	\begin{enumerate}[label=(\alph*)]
		\item $B$ is $L$-Lipschitz and
		$\la < \frac{1-3\a}{2L}$, or
		\item $B$ is $(1/L)$-cocoercive and $ \la <  \frac{1+\a}{2L}$.
	\end{enumerate}
	Given $x_0,x_{-1}\in\Hilbert$, define the sequence $(x_k)$ according to
	\begin{equation}\label{eq:inertial iteration}
	x_{k+1} := J_{\la A}\bigl( x_k - 2\la B(x_k) +\la B(x_{k-1}) + \a(x_k-x_{k-1}) \bigr) \\
	\end{equation}
	Then $(x_k)$ converges weakly to a point contained in $(A+B)^{-1}(0)$.
\end{corollary}

\begin{remark}
    Although Corollary~\ref{cor:inertial} establishes that inertia
    increases the range of admissible stepsizes when $B$ is
    cocoercive, it has the opposite effect when $B$ is merely
    monotone. A similar phenomena with Tseng's method was observed in
    \cite{boct2016inertial}.
\end{remark}

\begin{remark}
	By setting $B=0$ in Corollary~\ref{cor:inertial}, the
	scheme~\eqref{eq:inertial iteration} reduces to the classical
	\emph{inertial proximal algorithm} first considered in
	\cite{attouch01}. It is interesting to note that the proof presented here
	does not follow the technique from \cite{attouch01}, which is used in the analysis of most other first order inertial operator splitting methods
	\cite{pock:inertial,moudafi03,boct2016inertial}.
\end{remark}

\section{Three operator splitting}\label{s:3op}
In this section, we consider a structured three operator monotone inclusion. Specifically, we consider the inclusion
\begin{equation}\label{eq:3 operator}
\text{find}~x\in\Hilbert\text{ such that }0\in (A+B +C)(x),
\end{equation}
where $A\colon \Hilbert\setto\Hilbert$ is maximal monotone, $B\colon \Hilbert\to\Hilbert$ is monotone and $L_1$-Lipschitz, and $C\colon \Hilbert\to\Hilbert$ is $1/L_2$-cocoercive. This problem could be solved using the two operator splitting algorithm in Section~\ref{s:FoRB splitting} applied to $A$ and $(B+C)$, where we note that $(B+C)$ is $L$-Lipschitz continuous with $L=L_1+L_2$. Consequently, to apply Theorem~\ref{th:main}, the stepsize $\la$ should satisfy
 $$ \la < \frac{1}{2L} = \frac{1}{2L_1+2L_2}. $$
In this section, we show that this can be improved by exploiting the additional structure in \eqref{eq:3 operator}. Indeed, we propose a modification which only requires $\la>0$ to satisfy
  \begin{equation*}
  \la < \frac{2}{4L_1+L_2}= \frac{1}{2L_1+\frac{1}{2}L_2}.
  \end{equation*}
Given initial points $x_0,x_{-1}\in\Hilbert$, our modified scheme is given by
\begin{equation}\label{eq:3op iteration}
x_{k+1} = J_{\la A}\bigl(x_k - 2\la B(x_k) + \la B(x_{k-1}) - \la C(x_k) \bigr) \quad\forall k\in\mathbb{N}.
\end{equation}
In other words, the algorithm only uses a standard forward step of the operator $C$, as is employed in the forward-backward method~\eqref{eq:forward-backward}. For an algorithm for the case when $C$ is Lipschitz but not necessarily cocoercive, see \cite{ryu2019finding}.

We begin our analysis with the three operator analogue of Lemma~\ref{lem:decreasing sequence}.
\begin{lemma}\label{lem:3op decreasing sequence}
Let $x\in(A+B+C)^{-1}(0)$ and let the sequence $(x_k)$ be given by \eqref{eq:3op iteration}. Suppose $\la \in \left(0,\frac{2}{4L_1+L_2} \right) $. Then there exists an $\e>0$ such that, for all $k\in\mathbb{N}$, we have
	\begin{multline*}
	\n{x_{k+1}-x}^2 + 2\la\lr{B(x_{k+1})-B(x_k),x-x_{k+1}} + \left(\la L_1+\e \right)\n{x_{k+1}-x_k}^2\\ \leq \n{x_k-x}^2 + 2\la\lr{B(x_k)-B(x_{k-1}),x-x_{k}} + \la L_1\n{x_k-x_{k-1}}^2.
      \end{multline*}
\end{lemma}
\begin{proof}
	Since $0\in(A+B+C)(x)$, we have $-(B+C)(x)\in A(x)$. Combined with the monotonicity of $A$, this gives
	$$ 0\leq \lr{x_{k+1}-x_k + \la (B+C)(x_k) +  \la(B(x_k)-B(x_{k-1}))-\la (B+C)(x),x-x_{k+1}}, $$
	which we rewrite as
	\begin{multline}\label{eq:monotonicty inequality 3op}
	0\leq \lr{x_{k+1}-x_k,x-x_{k+1}} + \la\lr{B(x_k)-B(x),x-x_{k+1}} \\
	 +\la\lr{B(x_k)-B(x_{k-1}),x-x_{k}} +\la\lr{B(x_k)-B(x_{k-1}),x_k-x_{k+1}} \\
	 + \la\lr{C(x_k)-C(x),x-x_{k}} + \la\lr{C(x_k)-C(x),x_k-x_{k+1}}.
	\end{multline}
	The first through fourth terms can be estimated as in Lemma~\ref{lem:decreasing sequence}. Using $1/L_2$-cocoercivity of $C$, the fifth term can be estimated as
    $$ \lr{C(x_k)-C(x),x-x_{k}} \leq -\frac{1}{L_2}\n{C(x_k)-C(x)}^2,$$
    and the final term can estimated as
    \begin{align*}
     \lr{C(x_k)-C(x),x_k-x_{k+1}}
	 &\leq \n{C(x_k)-C(x)}\n{x_{k+1}-x_k} \\
	 &\leq  \frac{1}{L_2}\n{C(x_k)-C(x)}^2 + \frac{L_2}{4}\n{x_{k+1}-x_k}^2.
	\end{align*}
	Thus, altogether, \eqref{eq:monotonicty inequality 3op} implies that
	\begin{multline*}
	0\leq \n{x_k-x}^2 - \n{x_{k+1}-x_k}^2-\n{x_{k+1}-x}^2 - 2\la\lr{B(x_{k+1})-B(x_k),x-x_{k+1}}\\
	+ 2\la\lr{B(x_k)-B(x_{k-1}),x-x_{k}} + \la L_1\left(\n{x_k-x_{k-1}}^2+\n{x_{k+1}-x_k}^2\right) \\
	- \frac{2\la}{L_2}\n{C(x_k)-C(x)}^2 + \la \left(\frac{2}{L_2}\n{C(x_k)-C(x)}^2 + \frac{L_2}{2}\n{x_{k+1}-x_k}^2\right),
	\end{multline*}
	which, on rearranging, gives
	\begin{multline*}
	\n{x_{k+1}-x}^2 + 2\la\lr{B(x_{k+1})-B(x_k),x-x_{k+1}} + \left(1-\lambda L_1-\frac{\la L_2}{2}\right)\n{x_{k+1}-x_k}^2\\ \leq \n{x_k-x}^2 + 2\la\lr{B(x_k)-B(x_{k-1}),x-x_{k}} + \la L_1\n{x_k-x_{k-1}}^2.
	\end{multline*}
    The claimed inequality follows with
     $ \e := \left(1-\lambda L_1-\frac{\la L_2}{2}\right)-\la L_1 = 1-2\la L_1-\frac{\la L_2}{2}>0$.
\end{proof}

The following theorem is our main result regarding convergence of the three operator splitting scheme.
\begin{theorem}\label{th:3op}
	Let $A\colon \Hilbert\setto\Hilbert$ be maximally monotone, let
        $B\colon \Hilbert\to\Hilbert$ be monotone and $L_1$-Lipschitz, and
        let $C\colon \Hilbert\to\Hilbert$ be
        $1/L_2$-cocoercive. Suppose that
        $(A+B+C)^{-1}(0)\neq\varnothing$ and $\la
        \in\left(0,\frac{2}{4L_1+L_2}\right)$. Given $x_0,x_{-1}\in\Hilbert$,
        define the sequence $(x_k)$ according to
	\begin{equation*}
	x_{k+1} = J_{\la A}\bigl(x_k - 2\la B(x_k) + \la B(x_{k-1}) -\la C(x_k)  \bigr) \quad\forall k\in\mathbb{N}.
	\end{equation*}
	Then $(x_k)$ converges weakly to a point contained in $(A+B+C)^{-1}(0)$.
\end{theorem}
\begin{proof}
	The proof is more or less the same as Theorem~\ref{th:main} but uses Lemma~\ref{lem:3op decreasing sequence} in place of Lemma~\ref{lem:decreasing sequence}. The only other thing to check is that
	\begin{align*}
	\n{x_{k+1}-x}^2 + 2\la\lr{B(x_{k+1})-B(x_k),x-x_{k+1}} + \la L_1\n{x_{k+1}-x_k}^2
	\end{align*}
	is bounded from below by zero. To see this, observe that
	\begin{align*}
	2\la\lr{B(x_{k+1})-B(x_k),x-x_{k+1}}
	&\geq -2\la L_1 \n{x_{k+1}-x_k}\n{x_{k+1}-x} \\
	&\geq -\la L_1\left( \n{x_{k+1}-x_k}^2+ \n{x_{k+1}-x}^2 \right),
	\end{align*}
	and $1-\la L_1>0$ since $ \la < \frac{2}{4L_1+L_2} < \frac{1}{L_1}$.
\end{proof}

\section{Between forward-backward and forward-reflected-backward}\label{s:between}

In this section, we consider a variant of the
forward-reflected-backward method for a structured version of the
monotone inclusion \eqref{eq:zero 2sum} in a separable Hilbert space
$\Hilbert$. Precisely, we assume that the second operator
$B\colon\Hilbert\to\Hilbert$ is monotone and decomposable in the form
\begin{equation}
\label{eq:B decompose}B=\frac{1}{n}\sum_{i=1}^nB_i
\end{equation}
where $B_i\colon\Hilbert\to\Hilbert$ is $L$-Lipschitz continuous for $i=1,\dots,n$.
In what follows, we analyze the following iteration
\begin{equation}\label{eq:the iteration stoch}
\left\{\begin{array}{l}
\text{Choose~}i_k\text{~uniformly at random from~}\{1,\dots,n\} \\
x_{k+1} = J_{\la A}\bigl( x_k-\la B(x_k) - \la( B_{i_k}(x_k) - B_{i_k}(x_{k-1}) \bigr) \\
\end{array}\right.\quad\forall k\in\mathbb{N}.
\end{equation}
In other words, the term $B(x_k)-B(x_{k-1})$ inside the resolvent is
replaced by $B_{i_k}(x_{k}) - B_{i_k}(x_{k-1})$. Although it is unclear
if \eqref{eq:the iteration stoch} has any practical value as it
still requires one full evaluation of $B$ in every iteration, it is
surprising that such a small random perturbation still ensures its
(almost sure) convergence without cocoercivity. Indeed, without this
correction, the algorithm reduces to the forward-backward method
which, in general, need not converge in this setting. The fact that
\eqref{eq:the iteration stoch} converges suggests that the exact form
of the correction to values of $B$ may not be important.

In what follows, given a random variable $X$, $\mathbb{E}[X]$ denotes its expectation and $\mathbb{E}_k[X]$ denotes its conditional expectation with respect to the $\sigma$-algebra generated by the random variables $x_1,x_2,\dots,x_k$.

\begin{lemma}\label{lem:decreasing sequence stoch}
	Let $x\in(A+B)^{-1}(0)$ and let $(x_k)$ be given by \eqref{eq:the iteration stoch}.
	Suppose $\la \in \left(0,\frac{1}{2L}\right)$.
	Then there exists an $\e>0$ such that, for all $k\in\mathbb{N}$, we have
	\begin{multline}
	\ \|x_{k+1}-x\|^2+2\la\lr{B(x_{k+1})-B(x_k),x-x_{k+1}} + \left(\frac{1}{2}+\e\right)\|x_{k+1}-x_k\|^2\\
	\leq \|x_k-x\|^2  + 2\la\lr{B_{i_k}(x_k)-B_{i_k}(x_{k-1}), x-x_{k}}  +\frac{1}{2}\|x_{k}-x_{k-1}\|^2.
	\end{multline}
\end{lemma}
\begin{proof}
	By applying Proposition~\ref{prop:abs} with
	\begin{equation*}
	\begin{split}
	F&:=\la A\quad\\
	u&:=\la B(x) \quad
	\end{split}
	\quad
	\begin{split}
	d_1&:=x_k \\
	d_{2}&:=x_{k+1}
	\end{split}
	\quad
	\begin{split}
	u_{0}& :=\la B_{i_k}(x_{k-1}) \quad \\ u_1&:=\la B(x_k) \quad
	\end{split}
	\quad
	\begin{split}
	v_1 &:=\la B_{i_k}(x_k) \\ v_{2} &:= \la B(x_{k+1}) ,
	\end{split}
	\end{equation*}
	we obtain the inequality
	\begin{multline*}
	\n{x_{k+1}-x}^2 + 2\la\lr{B(x_{k+1})-B(x_k),x-x_{k+1}} + \n{x_{k+1}-x_k}^2\\
	\leq \n{x_k-x}^2 + 2\la\lr{B_{i_k}(x_k)-B_{i_k}(x_{k-1}),x-x_{k}} \\
	+ 2\la\lr{B_{i_k}(x_k)-B_{i_k}(x_{k-1}),x_k-x_{k+1}} - 2\la\lr{B(x_{k+1})-B(x),x_{k+1}-x}.
	\end{multline*}
	Since $B$ is monotone, the last term is nonnegative. Using Lipschitzness of $B_{i_k}$, the second last term can be estimated as
	\begin{equation*}
	\begin{split}
	\lr{B_{i_k}(x_k)-B_{i_k}(x_{k-1}),x_k-x_{k+1}}
	& \leq  L\n{x_k-x_{k-1}}\n{x_k-x_{k+1}} \\
	& \leq \frac{L}{2}\left(\n{x_k-x_{k-1}}^2+\n{x_k-x_{k+1}}^2\right).
	\end{split}
	\end{equation*}
	Thus, altogether, we obtain
	\begin{multline*}
	\n{x_{k+1}-x}^2 + 2\la\lr{B(x_{k+1})-B(x_k),x-x_{k+1}} + (1-\la L)\n{x_{k+1}-x_k}^2\\ \leq \n{x_k-x}^2 + 2\la\lr{B_{i_k}(x_k)-B_{i_k}(x_{k-1}),x-x_{k}} + \la L\n{x_k-x_{k-1}}^2.
	\end{multline*}
	The claimed inequality follows since $\la L<\frac{1}{2}$ and $1-\la L < \frac{1}{2}$.
\end{proof}

\begin{theorem}\label{th:main stoch}
	Suppose $\Hilbert$ is separable.
	Let $A\colon \Hilbert\setto\Hilbert$ be maximally monotone, and let $B\colon \Hilbert\to\Hilbert$ be monotone with $B=\sum_{i=1}^nB_i$ for $L$-Lipschitz continuous operators $B_i\colon\Hilbert\to\Hilbert$. Suppose that $(A+B)^{-1}(0)\neq\varnothing$ and that $\la \in \left(0,\frac{1}{2L}\right)$.
	Given $x_0,x_{-1}\in\Hilbert$, define the sequence $(x_k)$ according to \eqref{eq:the iteration stoch}.
	Then $(x_k)$ converges weakly almost surely to a point contained in $(A+B)^{-1}(0)$.
\end{theorem}
\begin{proof}
	Let $x\in(A+B)^{-1}(0)$ and let $(\varphi_k)\subseteq\mathbb{R}$ denote the sequence of random variables given by
	\begin{multline}\label{eq:phi sto}
	 \varphi_k :=  \|x_k-x\|^2  + 2\la\lr{B(x_k)-B(x_{k-1}), x-x_{k}}  +\frac{1}{2}\|x_{k}-x_{k-1}\|^2
	 \geq\frac{1}{2}\n{x_k-x}^2,
	\end{multline}
	where the latter inequality is due to \eqref{eq:lipschitzness of B
            for th:main}. Taking conditional expectation in
        Lemma~\ref{lem:decreasing sequence stoch} gives
	\begin{equation*}\label{eq:key inequality stoch}
	\mathbb{E}_k\left[ \varphi_{k+1} \right] +\e\mathbb{E}_k\left[ \|x_{k+1}-x_k\|^2 \right]
	\leq \varphi_k .
    \end{equation*}
    The supermartingale convergence theorem
    \cite[Theorem~1]{robbins1971convergence} then implies that, almost
    surely, $(\varphi_k)$ converges to a nonnegative-valued random
    variable $\varphi$ and that
    $\sum_{k=1}^\infty\mathbb{E}_k\bigl[\n{x_{k+1}-x_k}^2\bigr]<\infty$. The
    latter implies that $\n{x_{k+1}-x_k}^2\to0$ almost surely. From
    \eqref{eq:phi sto}, it then follows that $(x_k)$ is
    bounded almost surely and that $(\n{x_k-x}^2)$ converges almost
    surely to $\varphi$.

	Now, consider a realization $(x_k(\omega))$ of $(x_k)$ such
        that $\n{x_{k+1}(\omega)-x_k(\omega)}\to 0$ and
        $\varphi_k(\omega)\to\varphi(\omega)$ for some
        $\varphi(\omega)\geq 0$ (where $\varphi_k(\omega)$ denotes the
        corresponding realization of $\varphi_k$). Let
        $\overline{x}(\omega)$ be a sequential weak cluster point of
        the bounded sequence $(x_k(\omega))$. From \eqref{eq:the
            iteration stoch}, we have
	\begin{multline}\label{eq:inclusion in proof stoch}
	\frac{1}{\la}\big( x_{k-1}(\omega)-x_{k}(\omega) + \la\left( B(x_{k}(\omega))-B(x_{k-1}(\omega))\right)\\
	+ \la\left( B_{i_{k-1}}(x_{k-2}(\omega))-B_{i_{k-1}}(x_{k-1}(\omega))\right) \big) \in (A+B)(x_{k}(\omega))\qquad \forall k\geq 1.
	\end{multline}
	Since the graph of $A+B$ is demiclosed and $B_{1},\dots,B_n$
        are Lipschitz, taking the limit along a subsequence of
        $(x_k(\omega))$ which converges to $\overline{x}(\omega)$ in
        \eqref{eq:inclusion in proof stoch} yields
        $\bar{x}(\omega)\in(A+B)^{-1}(0)$. Altogether, we have that
        the weak sequential cluster points of $(x_k)$ are almost surely contained in
        $(A+B)^{-1}(0)$. An argument analogous to \cite[Proposition~2.3]{combettes2015stochastic} then shows that $(x_k)$ converges weakly almost surely to a $(A+B)^{-1}(0)$-valued random variable.
\end{proof}

\section{Concluding remarks}
In this work, we have proposed a modification of the forward-backward algorithm for finding a zero in the sum of two monotone operators which does not require cocoercivity.
To conclude, we outline three possible directions for further research into the method.

\textbf{Fixed point interpretations:} As the proof of the forward-reflected-backward method does not conform to the usual Krasnoselskii--Mann framework, it would be interesting to see if the method can be analyzed from the perspective of fixed point theory. To this end, consider the two operators $M,T\colon \Hilbert\times \Hilbert \to \Hilbert\times\Hilbert$ given by
$$M:=\begin{bmatrix} J_{\lambda A} & 0 \\ 0 & I \\  \end{bmatrix},\quad T:=\begin{bmatrix}
    I-2\la B & \la I \\ B & 0 \\ \end{bmatrix}.$$ By introducing the
auxiliary variable $u_{k+1}:=B(x_k)$, it is easy to see that
\eqref{eq:simplified iteration} may be expressed as the fixed point
iteration in $\Hilbert\times\Hilbert$ given by
$$ \binom{x_{k+1}}{u_{k+1}} = (M\circ T)\binom{x_k}{u_k}. $$
From the perspective of fixed point theory, it is not clear what properties the operator $M\circ T$ possesses which can be used to deduce convergence. For instance, although $M$ is \emph{firmly nonexpansive}, the operator $T$ need not be. A similar question regarding interpretations of the \emph{golden ratio algorithm}, for which the operator $M$ is of the same form,  was posed in \cite{malitsky2018golden}.

\textbf{Stochastic and coordinate extensions:} In large-scale problems, it is not always possible to evaluate the operator $B$ owing to its high computational cost. Two possibilities for reducing the computational requirements are  \emph{stochastic approximations} of $B(x_k)$ and \emph{block coordinate} variants of the algorithm. Both approaches work by employing low-cost approximation of $B(x_{k})$ in each iteration. It would be interesting to consider stochastic and coordinate extensions of the method proposed here.

\textbf{Acceleration schemes:} As explained in Section~\ref{s:intro},
the forward-reflected-backward method can be specialized to solve a
minimization problem involving the sum of two convex functions, one of
which is smooth. In 2007, in~\cite{nesterov07}, Nesterov exploited his
original idea from~\cite{nesterov1983method} to derive accelerated
proximal gradient methods that enjoy better complexity rates than the
standard forward-backward method.  It therefore seems reasonable that
the forward-reflected-backward method could be adapted to incorporate
a Nesterov-type acceleration.

\small

\paragraph{Acknowledgements.} The first author's research was supported by German Research Foundation grant SFB755-A4. The second author's research was supported in part by a Fellowship from the Alexander von Humboldt Foundation and in part by a Discovery Early Career Research Award from the Australian Research Council. The author would like to thank  associate editor and anonymous referees for their useful comments that have significantly
improved the quality of the paper.

\bibliographystyle{siam}
\bibliography{biblio}

\end{document}